\documentclass[reqno,11pt]{amsart}

\usepackage{amsmath,amsfonts,amssymb,amsthm,epsfig}
\usepackage{mathtools}
\usepackage[english]{babel}
\usepackage{tikz}
\usetikzlibrary{arrows}
\usepackage{comment}

 \allowdisplaybreaks
\numberwithin{equation}{section}

\font\script=rsfs10 at 11pt
\def\eps{\varepsilon}
\def\H{{\mbox{\script H}\,\,}}

\def\bar{{\rm bar}}
\def\R{\mathbb R}

\def\D{\mathcal D}

\def\bal{\begin{aligned}}
\def\eal{\end{aligned}}
\def\proofof#1{\begin{proof}[Proof of #1]}

\def\XXint#1#2#3{{\setbox0=\hbox{$#1{#2#3}{\int}$} \vcenter{\vspace{-1pt}\hbox{$#2#3$}}\kern-.5\wd0}}
\def\Xint#1{\mathchoice {\XXint\displaystyle\textstyle{#1}}{\XXint\textstyle\scriptstyle{#1}}{\XXint\scriptstyle\scriptscriptstyle{#1}}{\XXint\scriptscriptstyle\scriptscriptstyle{#1}}\!\int}
\def\intmed{\Xint{\hbox{---}}}

\pgfarrowsdeclare{<<<}{>>>}
{
\arrowsize=0.2pt
\advance\arrowsize by .5\pgflinewidth
\pgfarrowsleftextend{-4\arrowsize-.5\pgflinewidth}
\pgfarrowsrightextend{.5\pgflinewidth}
}
{
\arrowsize=0.2pt
\advance\arrowsize by .5\pgflinewidth
\pgfpathmoveto{\pgfpointorigin}
\pgfpathlineto{\pgfpoint{-1.2mm}{-.8mm}}
\pgfusepathqstroke
\pgfpathmoveto{\pgfpointorigin}
\pgfpathlineto{\pgfpoint{-1.2mm}{.8mm}}
\pgfusepathqstroke
}

\newcounter{mt}
\def\maintheorem#1#2#3{\par \medskip \noindent {\bf Theorem~\mref{#1}}~(#2).~{\it #3}\par}
\def\mref#1{\Alph{#1}}
\def\maintheoremdeclaration#1{\stepcounter{mt}\newcounter{#1}\setcounter{#1}{\arabic{mt}}}

\maintheoremdeclaration{main}

\newtheorem{theorem}{Theorem}[section]
\newtheorem{lemma}[theorem]{Lemma}
\newtheorem{prop}[theorem]{Proposition}

\newtheorem{defin}[theorem]{Definition}
\newtheorem{remark}[theorem]{Remark}

\begin{document}

\title{The sharp quantitative barycentric isoperimetric inequality for bounded sets}

\author{C. Gambicchia}
\author{A. Pratelli}

\begin{abstract}
We prove the sharp quantitative isoperimetric inequality in the case of the barycentric asymmetry, for bounded sets. This generalizes the $2$-D case recently proved in~\cite{BCH}.
\end{abstract}

\maketitle

\section{Introduction}

Quantitative isoperimetric inequalities have been attracting a huge interest in the last two decades. The basic question to be investigated is very simple. Namely, it is known that the set which minimizes the perimeter among those with fixed volume is the ball. But is true that a set which minimizes the perimeter up to a small error must be very close to a ball? Of course, one wants to give an affirmative answer, also providing a quantitative bound which correlates the perimeter gap of a set with its distance, in a suitable sense, to a ball. To make all this explicit, we begin by defining in the usual way the \emph{isoperimetric deficit} of a set $E\subseteq \R^N$, given by
\[
\delta(E) = \frac{P(E)-P(B(m))}{P(B(m))}\,.
\]
Here, by $B(m)$ we denote the ball centered at the origin and with mass $m=|E|$. We have to define now the \emph{asymmetry} of the set $E$, which is a measure of how much $E$ differs from being a ball. Notice that, while in the definition of the isoperimetric deficit we can use any ball of volume $m$, since they all have the same perimeter, in the definition of asymmetry one has to select a suitable ball, because of course we can guess that a set with a small deficit is very similar to some particular ball of volume $m$, but not necessarily to the one centered in the origin! There are several possible interesting definitions of asymmetry.\par

The one which has been more investigated is the so-called \emph{Fraenkel asymmetry}, defined by
\begin{equation}\label{deffras}
\lambda(E) = \inf \Bigg\{ \frac{\big|E\Delta (x+B(m))\big|}{|E|},\, x\in\R^N\Bigg\}\,,
\end{equation}
where we denote by ``$\Delta$'' the symmetric difference, that is, $A\Delta B= (A\setminus B) \cup (B\setminus A)$, and where the above infimum can actually be easily shown to be a minimum. With this notion of asymmetry, the sharp quantitative isoperimetric inequality reads as
\begin{equation}\label{ShQII}
\lambda(E) \leq C_F(N) \sqrt{\delta(E)}\,,
\end{equation}
where $C_F(N)$ is a geometric constant only depending on the dimension $N$, and the power $1/2$ is optimal. The above inequality has been proved with several different techniques starting from 2006 and it is now very well known, see for instance~\cite{FMP, FiMP, CL}.\par

However, this notion of asymmetry is not the only one which is meaningful, and there are at least two other ones which have been deeply investigated. A very natural one is the \emph{Hausdorff asymmetry}, given by
\[
\lambda_H(E) = \inf \Bigg\{ \frac{d_H\big(E,(x+B(m)\big)}{|E|},\, x\in\R^N\Bigg\}\,, 
\]
where $d_H$ is the Hausdorff distance. With this choice of asymmetry, the sharp inequality is
\begin{equation}\label{ShHaus}
\lambda_H(E) \leq C_H(N) \delta(E)^{p(N)}\,,
\end{equation}
which is valid for all \emph{convex sets} in $\R^N$ and where not only the constant $C_H(N)$ but also the power $p(N)$ depend on $N$. This inequality was proved in 1989 by Fuglede, who also found the exact formula for the sharp exponent $p(N)$, see~\cite{F89}. It is important to notice that, while the inequality~(\ref{ShQII}) is valid for all subsets of $\R^N$, the inequality~(\ref{ShHaus}) requires the set to be convex (or nearly spherical). This is not strange, since the inequality is easily seen to be false in general. It suffices to take a set $E$ given by a unit ball plus a very tiny ball at very large distance; of course, the isoperimetric deficit of $E$ can be made arbitrarily small as soon as the second ball is chosen small enough, while the Hausdorff asymmetry of $E$ is very close to the distance between the two balls, which can be made arbitrarily large. Speaking in general, it is more or less obvious that the Hausdorff distance is meaningful only when dealing with convex sets, or more generally sets which are known to have a special geometrical structure.\par

A last notion of asymmetry, which is the one we are interested in for this article, is the so-called \emph{barycentric asymmetry}, given by
\begin{equation}\label{defbaras}
\lambda_0(E) = \frac{\big|E\Delta (\bar(E)+B(m))\big|}{|E|}\,,
\end{equation}
where $\bar(E)$ is the barycenter of $E$ (which of course cannot be defined for any set, see Definition~\ref{defbar}). Before commenting on this asymmetry, we immediately point out that the corresponding sharp inequality, proved by Fuglede in 1993 (see~\cite{F93}) reads as
\begin{equation}\label{ShBar}
\lambda_0(E) \leq C_B(N) \sqrt{\delta(E)}\,,
\end{equation}
and is again valid for all \emph{convex sets} in $\R^N$. Observe that in this case, as in~(\ref{ShQII}) and unlike~(\ref{ShHaus}), the sharp exponent is again $2$.\par

Let us now discuss the barycentric asymmetry. First of all, we notice that it measures the distance between the set $E$ and a ball exactly as the Fraenkel asymmetry, that is, as the (rescaled) volume of the symmetric difference. The big difference is that, while with the Fraenkel asymmetry the ball is chosen so to minimise this volume, with the barycentric asymmetry the ball is simply the one centered in the barycenter. This is a strong and somehow ``arbitrary'' choice, but it is reasonable to guess that in most cases, if a set $E$ is very close to some ball, the center of this ball cannot be too far from the barycenter of $E$. Working with this asymmetry is then very handful, because the problem of ``choosing the correct ball'' is eliminated, and for instance this asymmetry is the easiest to use for a numerical approximation, since it is of course computationally much easier to calculate a barycenter and then a single volume of a symmetric difference, instead of calculating several such volumes in order to perform a minimization process. Also from a theoretical point of view, this asymmetry has been used intensively and it had a crucial role in the literature. Indeed, not only the asymmetry is simple to use as mentioned above because of the fact that the center of the ball should not be sought but is fixed by definition. But also, the Fuglede argument to prove~(\ref{ShBar}) can be easily repeated and generalised, since it basically consists in a clever way to reduce the proof to a boring but now standard calculation. For instance, in the celebrated paper~\cite{CL} in which the estimate~(\ref{ShQII}) for the Fraenkel asymmetry is shown, the authors are able to reduce themselves to a particular situation of ``nearly spherical sets'', and for these sets they argue as in Fuglede's proof of~(\ref{ShBar}), since for them the obvious approximation of the Frankel asymmetry from above with the barycentric asymmetry is not too bad and is strong enough to get the sharp estimate~(\ref{ShQII}).\par

A key observation is now to be done. As with the Hausdorff asymmetry, also with the barycentric one the inequality~(\ref{ShBar}) is surely not valid for all sets. The very same counterexample can be used. Indeed, putting together a unit ball with a second ball with radius $r\ll 1$ and distance $d\gg 1$ is again a set with a very small isoperimetric deficit. However, if $\omega_N d r^N >  3$, then the distance between the barycenter $\bar (E)$ and the center of the unit ball is bigger than $3$, and then the barycentric ball $\bar (E) + B(m)$ has no intersection with $E$, so that the barycentric asymmetry is $\lambda_0(E) = 2$, against the validity of~(\ref{ShBar}). However, while for the Hausdorff asymmetry to be meaningful it is pretty clear that convexity (or something close to it) is needed, the situation is less clear for the barycentric case. Indeed, the counterexample enlightens the fact that taking the barycenter in the definition of the asymmetry is not a meaningful choice for a fully general set. And on the other hand, Fuglede's proof shows that, instead, this is a smart choice for a convex set. However, from a geometrical point of view, one can imagine that the choice can be smart even for a case much more general than simply the convex sets. In other words, it is reasonable to hope that there is room for improvement.

This crucial observation has been the starting point of the very recent paper~\cite{BCH}. There, the authors notice that the ``obvious counterexample'', which is the one we already described above with the two balls, is not only not convex, but also not connected. This can be easily removed if the dimension is $N\geq 3$, because connecting the two balls with an incredibly thin cylinder with length $d$ and diameter much smaller than the radius $r$ of the small ball makes the set connected and has a negligible effect on the perimeter of $E$ and on its barycenter, so again $\lambda_0(E)\approx 2$ while $\delta(E) \approx 0$. But in dimension $N=2$ this is no more true, since a ``thin cylinder'' with length $d$ and incredibly small radius is actually a rectangle and it gives however a contribution at least $2d\gg 1$ to the perimeter, so in this case $\delta(E) \gg 1$. Therefore, it makes sense to ask oneself whether the barycentric inequality~(\ref{ShBar}) is valid for connected sets in $\R^2$, and the answer is actually positive.
\begin{theorem}[Bianchini--Croce--Henrot, 2023 (\cite{BCH})]\label{ThBCH}
There exists a constant $C_2$ such that, for every connected set $E\subseteq \R^2$, one has
\[
\lambda_0(E) \leq C_{BCH} \sqrt{\delta(E)}\,.
\]
\end{theorem}
This inspiring result was our starting point. Indeed, while as observed above the connectedness can be useful only in dimension $N=2$, we started wondering whether there is some other geometrical property, weaker than the convexity, which can make the barycentric asymmetry meaningful, or in other words, under which the inequality~(\ref{ShBar}) is still valid. We ended up noticing that such a simple geometrical property for any dimension $N$ exists, and it is the boundedness, which is of course extremely weaker than the convexity. More precisely, our main result is the following.
\maintheorem{main}{Quantitative barycentric isoperimetric inequality for bounded sets}{For every $N\geq 2$ and every $D>0$ there exists a constant $C(N,D)$ such that, for any set $E\subseteq\R^N$ with diameter less than $D |E|^{1/N}$, the inequality
\begin{equation}\label{ourineq}
\lambda_0(E) \leq C(N,D) \sqrt{\delta(E)}
\end{equation}
holds true.}
A couple of quick remarks have to be done. First of all, the fact that the diameter of $E$ must be bounded by $D |E|^{1/N}$, and not simply by $D$, is obvious; indeed, since we do not fix the volume of $E$, the diameter must be considered with respect to the ``expected diameter'' of the set, which is of order $|E|^{1/N}$. By trivial rescaling, the theorem is equivalent to say that the inequality~(\ref{ourineq}) is true for all sets $E$ of unit volume and diameter less than $D$. Second, the fact that the constant in~(\ref{ourineq}) depends on both $N$ and $D$ and not only on $N$ is again obvious. In fact, the information that $E$ is bounded is of no use without an estimate on the diameter, and of course the constant $C(N,D)$ must explode as $D\to +\infty$. A quick observation about the dependance of $C(N,D)$ on $D$ is contained in Remark~\ref{dependance}. Finally, it is interesting to compare Theorem~\mref{main} with the two preceding results mentioned above. We do this in the brief final Section~\ref{Sect:comparison}, where we observe that both these results readily follow from our one, and we comment on this.\par\bigskip

The plan of the paper is very simple. In Section~\ref{secprel} we collect the notation that we are going to use, and the few definitions and known results that will be needed later. Then, in Section~\ref{secproof} we present the proof of the main result. And finally, in Section~\ref{Sect:comparison}, we make a quick comparison between our result and the other ones on the same question.

We now conclude this introduction with a quick ``techcnical'' remark.
\begin{remark}
One can observe that, in the paper~\cite{BCH}, the authors where considering the Minkowski perimeter and the topological definition of connectedness, while one might prefer to use the standard definition of perimeter, and consequently using the measure theoretic definition of connectedness (see Definition~\ref{defconn}). However, this makes no practical difference at all; indeed, one can simply prove the sharp inequality just for smooth sets, for which the two definitions clearly coincide, and then argue by density.
\end{remark}

\subsection{\label{secprel}Preliminary definitions and results}

Through the paper, for any set $E\subseteq \R^N$ of locally finite perimeter, we will denote by $\partial^* E$ the reduced boundary of $E$, and by $P(E)=\H^{N-1}(\partial^* E)$ its perimeter (for these basic definitions one can see for instance~\cite{AFP}). The usual definition of \emph{barycenter} of a set $E$ is the following.
\begin{defin}\label{defbar}
Let $E\subseteq\R^N$ be a measurable set with strictly positive volume. We will call \emph{barycenter of $E$} the point
\[
\bar(E) = \intmed_E x \, dx
\]
whenever the integral is well defined. Notice that, for instance, the barycenter is always defined if $E$ is (essentially) bounded, so in particular if $E$ is convex and with finite measure.
\end{defin}
It is useful to recall also the standard measure theoretic definitions of diameter and of connectedness.

\begin{defin}[Diameter of sets]
The \emph{diameter} of a set $E\subseteq\R^N$ is defined as
\[
{\rm diam} (E) = \inf \bigg\{ d>0:\, \H^{2N} \Big( \Big\{(x,y)\in E\times E,\, |y-x| > d\Big\}\Big)=0\bigg\}\,.
\]
Equivalently, ${\rm diam}(E)$ is the supremum of the distances among pairs of Lebesgue points of $E$.
\end{defin}

\begin{defin}[Connected sets]\label{defconn}
A finite perimeter set $E\subseteq\R^N$ is said \emph{connected} if for every $F\subseteq E$ with $0< |F|<|E|$ one has
\[
P(E) < P(F) + P(E\setminus F)\,.
\]
\end{defin}

A very well-known, immediate consequence of the above definitions is the following one.

\begin{lemma}\label{connbdd}
Let $E\subseteq\R^2$ be a planar, connected set. Then, $P(E)\geq 2 {\rm diam} (E)$.
\end{lemma}

For our construction, we will need also the notion of $k$-symmetric sets and a property of $N$-symmetric ones.

\begin{defin}[$k$-symmetric sets]
A set $E\subseteq \R^N$ is said to be symmetric with respect to a hyperplane $\Pi$ if $E= R_\Pi (E)$, where $R_\Pi:\R^N\to\R^N$ is the reflection across to the hyperplane $\Pi$. A set $E\subseteq\R^N$ is said to be \emph{$k$-symmetric} for some integer $0\leq k\leq N$ if there exist $k$ orthogonal hyperplanes with respect to each of which $E$ is symmetric.
\end{defin}

When dealing with the Fraenkel asymmetry and symmetric sets, a simple but very useful observation is the following, for a proof see for instance~\cite{FMP}.

\begin{lemma}\label{lemmaNsymm}
Let $E\subseteq\R^N$ be an $N$-symmetric set with strictly positive volume, and assume just to fix the ideas that $E$ is symmetric with respect to each coordinate hyperplane $\Pi_i=\{x_i=0\}$, $1\leq i \leq N$. Then, the Fraenkel asymmetry of $E$ defined in~(\ref{deffras}) and the barycentric asymmetry of $E$ defined in~(\ref{defbaras}) satisfy the property
\[
\lambda(E) \leq \lambda_0(E) \leq 2^N \lambda(E)\,.
\]
\end{lemma}

We can quickly discuss the meaning of this lemma. The first inequality is obvious for any set $E$, since $\lambda(E)$ is the infimum of the volumes of the rescaled differences between $E$ and translations of the ball $B(m)$ with $|B(m)|=|E|$, while $\lambda_0$ corresponds to a particular translation. The interesting part is then the second inequality; notice that, as the above example with the union of two balls shows, the second inequality is false for a generic set $E$. The point of this lemma is then that considering the barycentric ball $\bar(E)+B(m)$ is an optimal choice (up to a multiplicative constant) in the special case of $N$-symmetric sets.

\section{\label{secproof}Proof of the main result}

This section is devoted to prove Theorem~\mref{main}. This will be obtained as an immediate consequence of the following technical property.

\begin{prop}\label{allhere}
If for some $1\leq k\leq N$ there exist constants $C_k(N,D)\geq 2^N C_F(N)$ such that the inequality
\begin{equation}\label{inductive}
\lambda_0(E) \leq C_k(N,D) \sqrt{\delta(E)}
\end{equation}
holds true for every $k$-symmetric set $E$ with diameter less than $D |E|^{1/N}$, then the same is true also for any $(k-1)$-symmetric set $E$ with diameter less than $D |E|^{1/N}$ and with constant
\begin{equation}\label{defCk-1}
C_{k-1}(N,D) = C_3 C_k(N,3D) D\,,
\end{equation}
where $C_3=C_3(N)$ is defined in~(\ref{defC3}). In particular, also $C_{k-1}(N,D)\geq 2^N C_F(N)$.
\end{prop}

We can immediately see that our main result readily follows from the above proposition.

\proofof{Theorem~\mref{main}}
If $E$ is a $N$-symmetric set with diameter less than $D |E|^{1/N}$, then putting together the standard quantitative isoperimetric inequality~(\ref{ShQII}) and Lemma~\ref{lemmaNsymm} we get
\[
\lambda_0(E) \leq 2^N \lambda(E) \leq 2^N C_F(N) \sqrt{\delta(E)}\,,
\]
thus the inequality~(\ref{inductive}) is true for $N$-symmetric sets with constant
\[
C_N(N,D) = 2^N C_F(N)\,.
\]
Applying then $N$ times Proposition~\ref{allhere}, we get the validity of~(\ref{inductive}) for $0$-symmetric sets with constant
\begin{equation}\label{defCND}
C_0(N,D)= 2^N C_F(N) C_3^N 3^{\frac{N(N-1)}2} D^N\,.
\end{equation}
Since any set is $0$-symmetric, we have proved~(\ref{ourineq}) for a generic set $E\subseteq \R^N$ with diameter less than $D |E|^{1/N}$ and with constant $C(N,D)=C_0(N,D)$. The proof is then concluded.
\end{proof}

The proof of Proposition~\ref{allhere} will take this whole section. We start by considering a $(k-1)$-symmetric set $E$, with diameter less than $D|E|^{1/N}$. We assume that
\begin{align}\label{assonE}
|E| = 1\,, && \bar(E)=O\,, && R_{\Pi_i}(E)=E \quad \forall\, N-k+1<i\leq N\,.
\end{align}
Notice that, since our result is scaling invariant, we are allowed to fix the volume of $E$, and we have fixed $|E|=1$ just for the ease of notation. Again for convenience, and up to a translation and a rotation, we can assume without loss of generality that the barycenter is in the origin and that $E$ is symmetric with respect to the last $k-1$ coordinate hyperplanes. Our plan is to modify $E$ so to become symmetric also with respect to the hyperplane $\Pi_1=\{x_1=0\}$, so becoming $k$-symmetric. To do so, we begin by calling
\begin{align}\label{defE+-}
E^+=E \cap \{x_1>0\}\,, && E^-= E\cap\{x_1<0\}
\end{align}
the two parts in which $E$ is divided by the hyperplane $\{x_1=0\}$, and
\begin{equation}\label{defeps}
\eps = \bigg|\frac 12 - |E^+| \bigg|\,.
\end{equation}
The number $\eps$ determines how different are the volumes of the two parts $E^+$ and $E^-$, in particular $\eps=0$ only if the hyperplane $\Pi_1$ divides $E$ in two parts of equal volume. The diameter constraint allows to give a simple estimate of $\eps$ in terms of $\delta(E)$ as follows.

\begin{lemma}\label{lemmaepssqrtdelta}
There exists a constant $C_1$, only depending on $N$ and defined in~(\ref{defC1}), such that
\[
\eps \leq C_1 D \sqrt{\delta(E)}\,.
\]
\end{lemma}
\begin{proof}
Let us call $B_F$ a ``Fraenkel ball'', that is, a ball realizing the infimum in~(\ref{ShQII}). The fact that such a ball exists is very simple; indeed, taking a ball centered in a Lebesgue point of $E$ ensures that $\lambda(E)$ is strictly smaller than $2$ (which corresponds to a ball having negligible intersection with $E$). We can than take a sequence of points $x_j$ with the property that the ratio $|E\Delta ( x_j + B)|/|E|$ converges to $\lambda(E)$. This sequence must be bounded since otherwise the ratio would converge to $2$, while $\lambda(E)<2$, and then any limit point of the sequence $\{x_j\}$ corresponds to a Fraenkel ball. Let us now call
\begin{equation}\label{defepsF}
\eps_F = \bigg|\big| B_F \cap \{x_1>0\} \big| - \frac 12 \bigg|
\end{equation}
the difference between the volume of $B_F$ and $1/2$. Then, $\eps_F=0$ if and only if the center of $B_F$ lies in the hyperplane $\{x_1=0\}$. More in general, if we call $d$ the distance between this hyperplane and the center of $B_F$, since the radius of $B_F$ is $1/\omega_N^{1/N}$, we have
\begin{equation}\label{spum}
\eps_F \leq \frac{\omega_{N-1} }{\omega_N^{\frac{N-1}N}}\, d\,.
\end{equation}
Let us then call for brevity  $G^+ = B_F\setminus E$ and $G^- = E \setminus B_F$, which are two sets with volume $|B_F\Delta E|/2$ each, and notice that
\[\begin{split}
\frac{\omega_N^{\frac{N-1}N}}{\omega_{N-1}}\, \eps_F &\leq  d = \bigg| \int_{B_F} x_1 \,dx\bigg|
= \bigg| \int_E x_1 \,dx+\int_{G^+} x_1 \,dx-\int_{G^-} x_1 \,dx\bigg|
= \bigg|\int_{G^+} x_1 \,dx-\int_{G^-} x_1 \,dx\bigg|\\
&\leq D \big(|G^+|+|G^-|\big) = D |B_F\Delta E| = D \lambda(E)\,,
\end{split}\]
where we have used~(\ref{spum}) and the facts that $\bar(E)=O$, the diameter of $E$ is at most $D$, and $B_F$ is a Fraenkel ball for $E$. We have thus found that
\[
\eps_F \leq \frac{D \omega_{N-1}}{\omega_N^{\frac{N-1}N}}\,  \lambda(E)  \,.
\]
Moreover, notice that by~(\ref{defeps}) and~(\ref{defepsF})
\[
\lambda(E) = |E\Delta B_F| \geq \Big|\big| E \cap \{x_1>0\}\big| - \big| B_F \cap \{x_1>0\}\big|\Big| \geq |\eps - \eps_F|\,.
\]
The last two inequalities, together with~(\ref{ShQII}) and with the fact that $D\geq 2 \omega_N^{-1/N}$, imply that
\[\begin{split}
\eps &\leq \eps_F + |\eps - \eps_F| \leq \frac{D \omega_{N-1}}{\omega_N^{\frac{N-1}N}}\,  \lambda(E) + \lambda(E)
= \lambda(E) \bigg( \frac{D \omega_{N-1}}{\omega_N^{\frac{N-1}N}}  + 1 \bigg)\\
&\leq  C_F(N) \Bigg( \frac{D \omega_{N-1}}{\omega_N^{\frac{N-1}N}}  + 1 \Bigg) \sqrt{\delta(E)}
\leq  C_F(N) D  \Bigg( \frac{\omega_{N-1}}{\omega_N^{\frac{N-1}N}}  + \frac{\omega_N^{1/N}}2 \Bigg) \sqrt{\delta(E)}
\,.
\end{split}\]
This concludes the thesis with
\begin{equation}\label{defC1}
C_1 =  C_F(N)\Bigg( \frac{\omega_{N-1}}{\omega_N^{\frac{N-1}N}}  + \frac{\omega_N^{1/N}}2 \Bigg)\,.
\end{equation}
\end{proof}

\begin{remark}\label{epsnotbad}
An obvious consequence of the above lemma is that, if $8\eps \geq \lambda_0(E)$, then~(\ref{ourineq}) is true with $8C_1 D$ in place of $C(N,D)$. In other words, in proving Theorem~\mref{main} we can assume for free that $\eps\leq \lambda_0(E)/8\leq 1/4$. By~(\ref{defeps}), this means that the volumes of $E^+$ and $E^-$ are between $1/4$ and $3/4$, that is, the hyperplane $\Pi_1$ is not necessarily dividing $E$ in two parts of equal volume, but they are also not too much different.
\end{remark}

We collect now an extremely simple observation, which will be used later.
\begin{lemma}\label{trilem}
Let $G,\, H$ and $\widetilde H$ be three sets so that $|\widetilde H|=|G|$ and $|H\Delta \widetilde H| = \big| |H| - |G|\big|$. Then
\[
2 |G \Delta H|  \geq |G\Delta \widetilde H| \geq \frac {2|G|}{D'} \big| \bar(G) - \bar (\widetilde H)\big|\,,
\]
where we call $D'$ the diameter of $G\cup \widetilde H$.
\end{lemma}
\begin{proof}
We notice that, since $|\widetilde H|=|G|$, then $|H\Delta \widetilde H| = \big| |H| - |G|\big|$ is equivalent to say that $|H\Delta \widetilde H| = \big| |H| - |\widetilde H|\big|$, which in turn is equivalent to say that either $H\subseteq \widetilde H$ or $\widetilde H\subseteq H$. The first inequality of the claim readily follows since
\[
|G\Delta \widetilde H| \leq |G\Delta H| + |H \Delta \widetilde H| = |G\Delta H| + \big| |H| - |G|\big|  \leq 2 |G\Delta H|\,.
\]
Concerning the second inequality, it is sufficient to recall that $|G|=|\widetilde H|$ and that the distance between any point of $G$ and any point of $\widetilde H$ is at most $D'$, getting the thesis since
\[\begin{split}
\big| \bar(G) - \bar (\widetilde H)\big| &= \frac 1{|G|}\, \bigg| \int_G x\, dx - \int_{\widetilde H} x\, dx\bigg|
= \frac 1{|G|}\, \bigg| \int_{G\setminus \widetilde H} x\, dx - \int_{\widetilde H\setminus G} x\, dx\bigg|
\leq \frac {D'}{|G|}\, \big| G\setminus \widetilde H \big|\\
&= \frac {D'}{2|G|}\, \big|  G\Delta \widetilde H \big|\,.
\end{split}\]
\end{proof}

Let us now present the proof of Proposition~\ref{allhere}

\proofof{Proposition~\ref{allhere}}
Let us take any $(k-1)$-symmetric set $E$ satisfying~(\ref{assonE}). We restrict ourselves to consider the case $\eps\leq \lambda_0(E)/8\leq 1/4$. Indeed, if this is not the case, as noticed in Remark~\ref{epsnotbad} we already have the validity of~(\ref{ourineq}) for $E$ with constant $8C_1 D$, and then also~(\ref{inductive}) is valid for the $(k-1)$-symmetric set $E$ with constant $C_{k-1}(N,D)$ as soon as $C_{k-1}(N,D) \geq 8 C_1 D$. \par

Let us call again $E^\pm$, as in~(\ref{defE+-}), the two parts of $E$ with $\{x_1\gtrless 0\}$. Let also
\[
P = \int_{E^+} x\, dx= |E^+| \bar(E^+)\,.
\]
The point $P$ is of course in the halfspace $\{x_1>0\}$. Moreover, we call
\[
\eta = \big| P - P_1\,{\rm e}_1\big| = \frac 12\, \big|  P + R_{\Pi_1} (P) \big| = |\widehat P|
\]
the distance between the origin and the point $\widehat P$, which is the projection of $P$ onto the hyperplane $\Pi_1$. Let us now define the two sets
\begin{align}\label{defE'''}
E' = E^+ \cup R_{\Pi_1} (E^+) \,, && E'' = E^- \cup R_{\Pi_1} (E^-) \,.
\end{align}
Notice that $E'$ and $E''$ are two sets which are both symmetric with respect to $\Pi_1$, and by~(\ref{assonE}) they keep the symmetries with respect to $R_{\Pi_i}$ for $N-k+1<i\leq N$, so they are actually $k$-symmetric. Moreover, calling $P'=\bar(E')$, we have
\[
|P'| = \big|\bar(E')\big| = \bigg|\intmed_{E'} x\,dx \bigg| = \frac{\big|P + R_{\Pi_1}(P)\big|}{|E'|}=  \frac{\widehat P}{|E^+|} =\frac \eta {|E^+|} =:\eta' \,,
\]
and thanks to the assumption that $\eps\leq 1/4$ we have
\[
\frac 43\, \eta < \eta' < 4\eta\,.
\]
To help reading the proof, Figure~\ref{nicepic} shows a possible set $E$ in the left side, and the corresponding set $E'$ in the right.
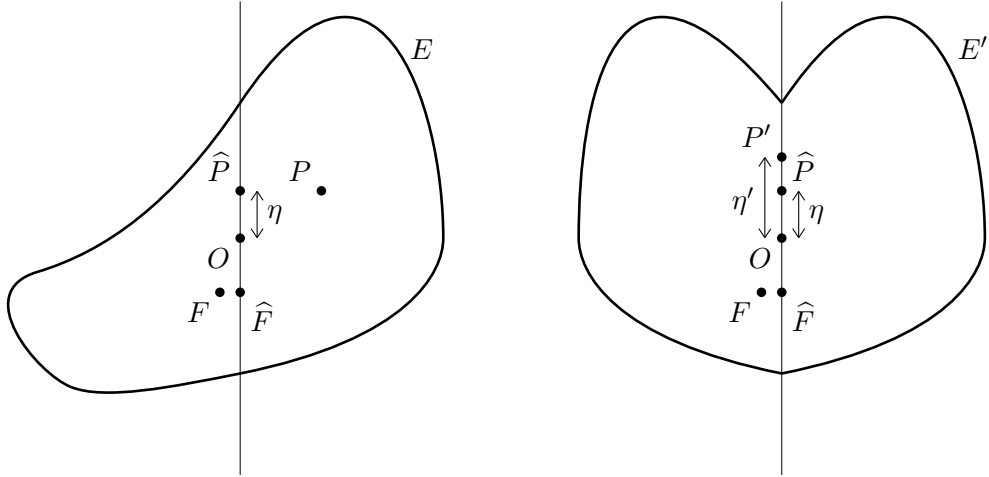
\begin{figure}[thbp]
\begin{tikzpicture}[>=>>>, scale=0.9] 
\fill (0,0) circle (2pt);
\draw (0,0) node[anchor=north east] {$O$};
\fill (-.3,-.8) circle (2pt);
\draw (-.3,-.8) node[anchor=north east] {$F$};
\fill (0,-.8) circle (2pt);
\draw (0,-.8) node[anchor=north west] {$\widehat F$};
\draw[line width=0.25pt] (0,3.5) -- (0,-3.5);
\fill (1.2,0.7) circle (2pt);
\draw (1.2,0.7) node[anchor=south east] {$P$};
\fill (0,0.7) circle (2pt);
\draw (0,0.7) node[anchor=south east] {$\widehat P$};
\draw (3,2.5) node[anchor=south east] {$E$};
\draw [<->] (.25,0) -- (0.25,0.7);
\draw (.25,0.35) node[anchor=west] {$\eta$};
\draw[line width=1pt] (0,-2) .. controls (2.5,-1.5) and (3,-.5) .. (3,0) .. controls (3,2) and (2,5) .. (0,2) .. controls (-1, 0.5) and (-2, -0.2) .. (-3,-0.5) .. controls (-4,-0.8) and (-3,-2) .. (-2.5,-2.2) .. controls (-2,-2.4) and (-1,-2.2) .. (0,-2);
\fill (8,0) circle (2pt);
\draw (8,0) node[anchor=north east] {$O$};
\fill (7.7,-.8) circle (2pt);
\draw (7.7,-.8) node[anchor=north east] {$F$};
\fill (8,-.8) circle (2pt);
\draw (8,-.8) node[anchor=north west] {$\widehat F$};
\draw[line width=0.25pt] (8,3.5) -- (8,-3.5);
\fill (8,1.2) circle (2pt);
\draw (8,1.2) node[anchor=south east] {$P'$};
\fill (8,0.7) circle (2pt);
\draw (8,0.7) node[anchor=south west] {$\widehat P$};
\draw (11.2,2.5) node[anchor=south east] {$E'$};
\draw [<->] (8.25,0) -- (8.25,0.7);
\draw (8.25,0.35) node[anchor=west] {$\eta$};
\draw [<->] (7.75,0) -- (7.75,1.2);
\draw (7.75,0.55) node[anchor=east] {$\eta'$};
\draw[line width=1pt] (8,-2) .. controls (10.5,-1.5) and (11,-.5) .. (11,0) .. controls (11,2) and (10,5) .. (8,2);
\draw[line width=1pt] (8,-2) .. controls (5.5,-1.5) and (5,-.5) .. (5,0) .. controls (5,2) and (5.5,5) .. (8,2);
\end{tikzpicture}
\caption{A possible set $E$ in the proof of Proposition~\ref{allhere} and the corresponding set $E'$ with the quantities $\eta$ and $\eta'$.}\label{nicepic}
\end{figure}
Notice that, as in the figure, the barycenter $P'$ of $E'$ (which by construction belongs to $\Pi_1$) has distance $\eta'$ from the origin which is larger than $\eta$ but controlled by it. In the very same way we define the barycenter $P''$ of $E''$ and the constant $\eta'' \in (\frac 43\,\eta,4\eta)$.\par

We define now $B_\eta'$ and $B_\eta''$ the balls centered in $P'$ and $P''$ and with volumes $|E'|$ and $|E''|$, thus $1\pm 2\eps$. Moreover, we denote by $\widetilde B'$ and $\widetilde B''$ the two balls centered in $P'$ and $P''$ with unit volume, and by $B$ the ball centered in the origin with unit volume. We can then easily calculate, keeping in mind the definitions and the symmetries of the sets,
\[\begin{split}
\lambda_0(E') &= \big| E' \Delta B_\eta'\big|
= 2 \Big| \big(E' \Delta B_\eta'\big)\cap \{x_1>0\}\Big|\\
&\geq 2 \Big| \big(E' \Delta B\big)\cap \{x_1>0\}\Big|-2 \Big| \big(B \Delta B_\eta'\big)\cap \{x_1>0\}\Big|\\
&=2 \Big| \big(E \Delta B\big)\cap \{x_1>0\}\Big|-\big| B \Delta B_\eta'\big|
\geq 2 \Big| \big(E \Delta B\big)\cap \{x_1>0\}\Big|-\big| B \Delta \widetilde B'\big| - \big| \widetilde B' \Delta B_\eta'\big|\\
&\geq 2 \Big| \big(E \Delta B\big)\cap \{x_1>0\}\Big|-\frac{4\omega_{N-1}}{\omega_N^{\frac{N-1}N}}\, \eta - 2\eps\,.
\end{split}\]
Concerning the last inequality, one only has to notice that $| \widetilde B' \Delta B_\eta'|=2\eps$ by construction, since they are two balls with the same center, while $B \Delta \widetilde B'$ is contained in a cylinder of radius $\omega_N^{-1/N}$ and height $\eta'<4\eta$. Adding this inequality to the symmetric one obtained with $E''$ in place of $E'$, and observing that
\[
\Big| \big(E \Delta B\big)\cap \{x_1>0\}\Big|+\Big| \big(E \Delta B\big)\cap \{x_1<0\}\Big| = \big|E \Delta B\big| = \lambda_0(E)\,,
\]
we obtain then
\[
\lambda_0(E) \leq \frac{\lambda_0(E') + \lambda_0(E'')}2 + \frac{4\omega_{N-1}}{\omega_N^{\frac{N-1}N}}\, \eta + 2\eps\,,
\]
which, since we are assuming $\eps\leq \lambda_0(E)/8$, implies
\begin{equation}\label{usend}
\lambda_0(E) \leq \lambda_0(E') + \lambda_0(E'') + \frac{8\omega_{N-1}}{\omega_N^{\frac{N-1}N}}\, \eta \,.
\end{equation}
Now, let us call $B_F$ a Fraenkel ball for $E$, and call $F$ its center; moreover, call $\widehat F$ and $\widehat P$ the projections of $F$ and $P$ onto $\Pi_1$. By construction,
\begin{align}\label{guaqui}
P' = \frac{\widehat P}{|E|^+}\,, && P'' = -\frac{\widehat P}{|E|^-}\,.
\end{align}
We assume now that, as shown in Figure~\ref{nicepic},
\begin{equation}\label{awlog}
\widehat P \cdot \widehat F \leq 0\,,
\end{equation}
and we will work on the half-space $\{x_1>0\}$. By~(\ref{guaqui}), if~(\ref{awlog}) is not true then the very same construction can be done in the half-space $\{x_1<0\}$. Let us now call $\widehat B$ the ball centered in $F$ and such that
\begin{equation}\label{constr}
\big| \widehat B \cap \{x_1>0\}\big| = \big| E \cap \{x_1>0\}\big| \,.
\end{equation}
We can now apply Lemma~\ref{trilem} to the three sets
\begin{align*}
G = E\cap \{x_1>0\} \,, && H = B_F \cap \{x_1>0\}\,, && \widetilde H= \widehat B\cap \{x_1>0\}\,.
\end{align*}
Notice that this is admissible because the assumption $|\widetilde H|=|G|$ is true by~(\ref{constr}), and the assumption $|H\Delta \widetilde H| = \big| |H| - |G|\big|$ is true because one of the two sets $H$ and $\widetilde H$ is contained into the other one, since they are the intersection of an half-space with two balls with the same center. Putting then together the claim of Lemma~\ref{trilem} and the quantitative isoperimetric inequality~(\ref{ShQII}), keeping in mind that $B_F$ is a Fraenkel ball, we get
\[
C_F(N) \sqrt{\delta(E)} \geq \lambda(E) = |E \Delta B_F|
\geq \Big|\big(E \Delta B_F\big) \cap \{x_1>0\}\Big|
= |G \Delta H|
\geq \frac {|G|}{D'} \big| \bar(G) - \bar (\widetilde H)\big|\,,
\]
where $D'$ is the diameter of $G\cup \widetilde H$. Keep in mind that the diameter of $E$ is less than $D|E|^{1/N}=D$, thus any point of $G$ has distance at most $D$ from the origin. On the other hand, the distance of $F$ from the origin is at most $D + \omega_N^{-\frac 1N}\leq 2D$, because otherwise the Fraenkel ball (which has radius $\omega_N^{-\frac 1N}$) has empty intersection with $E$, and this is impossible. Finally, the radius of $\widehat B$ is less than $4D$, because otherwise $\widetilde H$ contains a whole ball of radius $D$, so its volume is more than $\omega_N D^N\geq 1$, and this is impossible because this volume equals the volume of $E^+$, that is between $1/4$ and $3/4$. Summarizing, the diameter $D'$ of $G\cup \widetilde H$ is less than $7D$. One can observe that we have been very rough in doing this estimate, but a more careful estimate would be in any case of the form $D' \leq c D$ with $c\geq 1$, so the dependence on $D$ would be with the same exponent and only the multiplicative constant would be smaller than $7$. The above estimate gives then
\begin{equation}\label{atlo}
C_F(N) \sqrt{\delta(E)} \geq 7\, \frac{|G|}D \,\big| \bar(G) - \bar (\widetilde H)\big|\,.
\end{equation}
Let us now consider the distance between the barycenters of $G$ and of $\widetilde H$. Since $G=E^+$, by definition we have
\[
\bar(G) = \intmed_{E^+} x\,dx = \frac 1{|E^+|} \int_{E^+} x\, dx =  \frac P{|E^+|}\,.
\]
Therefore, the projection of $\bar(G)$ onto $\Pi_1$ coincides with $\widehat P / |E^+|$, since $\widehat P$ is the projection of $P$ onto $\Pi_1$. Concerning $\widetilde H$, instead, keep in mind that $\widehat B$ is a ball centered in $F$, and $\widetilde H= \widehat B \cap \{x_1>0\}$. As a consequence, the projection of $\bar (\widetilde H)$ onto $\Pi_1$ is the same as the projection of $F$ onto $\Pi_1$, which is the point $\widehat F$. Since the distance between $\bar(G)$ and $\bar(\widetilde H)$ is larger than the distance between their projections on $\Pi_1$, and keeping in mind~(\ref{awlog}), we get
\[
\big| \bar(G) - \bar (\widetilde H)\big| \geq \bigg| \frac{\widehat P}{|E^+|} - \widehat F\bigg|\geq  \bigg| \frac{\widehat P}{|E^+|}\bigg|
=  \frac{|\widehat P|}{|E^+|} = \frac \eta{|E^+|}\,.
\]
Finally, inserting this estimate into~(\ref{atlo}), and keepind in mind again that $G=E^+$, we get
\[
C_F(N) \sqrt{\delta(E)} \geq 7\, \frac \eta D\,.
\]
Let now $\widetilde E$ be the set which maximizes $\lambda_0$ among $E'$ and $E''$, so $\widetilde E= E'$ if $\lambda_0(E')\geq \lambda_0(E'')$, and otherwise $\widetilde E=E''$. Notice that by construction $\widetilde E$ is a $k$-symmetric set, and since $|\widetilde E|\geq 1/2$, then its diameter is less than $2D \leq 3D |\widetilde E|^{1/N}$. Therefore, putting this last estimate into~(\ref{usend}) and applying~(\ref{inductive}) on $\widetilde E$, we get
\begin{equation}\label{abouttoend}\begin{split}
\lambda_0(E) &\leq \lambda_0(E') + \lambda_0(E'') + \frac{8\omega_{N-1}}{\omega_N^{\frac{N-1}N}}\, \eta
\leq 2\lambda_0(\widetilde E) + \frac{8\omega_{N-1}}{7\omega_N^{\frac{N-1}N}}\, D C_F(N) \sqrt{\delta(E)}\\
&\leq 2 C_k(N,3D) \sqrt{\delta(\widetilde E)} + \frac{8\omega_{N-1}}{7\omega_N^{\frac{N-1}N}}\, D C_F(N) \sqrt{\delta(E)} \,.
\end{split}\end{equation}
To conclude, we need now to estimate $\delta(\widetilde E)$ in terms of $\delta(E)$. Since $B_\eta'$ and $B_\eta''$ are two balls with the same volumes as $E'$ and $E''$, and since these volumes are both larger than $1/2$, we have
\[\begin{split}
\delta(E') +\delta(E'') &= \frac{P(E') - P(B_\eta')}{P(B_\eta')}+\frac{P(E'') - P(B_\eta'')}{P(B_\eta'')}
= \frac{P(E') - P(B_\eta')}{N\omega_N^{1/N} |E'|^{\frac{N-1}N}}+\frac{P(E'') - P(B_\eta'')}{N\omega_N^{1/N} |E''|^{\frac{N-1}N}}\\
&\leq 2^{-\frac{N-1}N} \, \frac{P(E')+P(E'')-(P(B_\eta')+P(B_\eta''))}{N\omega_N^{1/N}}\,.
\end{split}\]
Using again that $1-2\eps\geq 1/2$, we get
\[\begin{split}
P(B_\eta')+P(B_\eta'') &= N\omega_N^{1/N} \Big(|B'|^{\frac{N-1}N}+|B''|^{\frac{N-1}N}\Big)
= N\omega_N^{1/N} \Big(\big(1 + 2\eps\big)^{\frac{N-1}N}+\big(1- 2\eps\big)^{\frac{N-1}N}\Big)\\
&\geq N\omega_N^{1/N} \bigg( 2 - 2^{3+\frac 1N} \frac {N-1}{N^2}\,\eps^2\bigg)
=2P(B) - 2^{3+\frac 1N}N\omega_N^{1/N}\, \frac {N-1}{N^2}   \,\eps^2\,.
\end{split}\]
Therefore, since $P(E')+P(E'') \leq 2P(E)$, we can continue the above estimate as
\[\begin{split}
\delta(E') +\delta(E'')&\leq 2^{-\frac{N-1}N} \, \frac{2P(E)-2P(B) +2^{3+\frac 1N}N\omega_N^{1/N}\, \frac {N-1}{N^2}   \,\eps^2 }{N\omega_N^{1/N}}\\
&=2^{-\frac{N-1}N} \bigg( 2\delta(E)+2^{3+\frac 1N}\, \frac {N-1}{N^2}   \,\eps^2\bigg)
= 2^{1/N} \delta(E) + 2^{2+\frac 2N}\, \frac{N-1}{N^2}\,\eps^2\\
&\leq \Big(2^{1/N}  + 2^{2+\frac 2N}\, \frac{N-1}{N^2}\,C_1^2 D^2 \Big)\delta(E)
\leq C_2^2 D^2  \delta(E)
\end{split}\]
where we have also used Lemma~\ref{lemmaepssqrtdelta}, the fact that $D\geq 2 \omega_N^{1/N}$ by the isodiametric inequality, and where we have set
\begin{equation}\label{defC2}
C_2 = \sqrt{2^{\frac 1N-2} \omega_N^{-\frac 2N} + 2^{2+\frac 2N}\, \frac{N-1}{N^2}\,C_1^2 }\,.
\end{equation}
Notice that $C_2$ is a purely geometric constant only depending on $N$, since so is $C_1$.\par
We are now about to conclude. Indeed, since $\widetilde E$ is one between $E'$ and $E''$, the last estimate implies
\[
\delta(\widetilde E)\leq C_2^2 D^2 \delta(E)\,,
\]
which inserted in~(\ref{abouttoend}) gives
\[
\lambda_0(E) \leq \bigg(2 C_k(N,3D) C_2  + \frac{8\omega_{N-1}}{7\omega_N^{\frac{N-1}N}}\,  C_F(N) \bigg)\, D \sqrt{\delta(E)} \,.
\]
Keep in mind that this estimate holds under the assumption that $\eps\leq \lambda_0(E)/8\leq 1/4$, while otherwise we have $\lambda_0(E) \leq 8C_1 D\sqrt{\delta(E)}$. Therefore, we have shown that~(\ref{inductive}) is valid for any $(k-1)$-symmetric set $E$, as soon as we define
\begin{equation}\label{quick}
C_{k-1}(N,D) \geq  8C_1  D \vee \bigg(2 C_k(N,3D) C_2  + \frac{8\omega_{N-1}}{7\omega_N^{\frac{N-1}N}}\,  C_F(N) \bigg)\, D\,.
\end{equation}
We can then finally set
\begin{equation}\label{defC3}
C_3 = \max \bigg\{ \frac{8C_1}{2^N C_F(N)} ,\, 2 C_2 + \frac{2^{3-N} \omega_{N-1}}{7\omega_N^{\frac{N-1}N}},\, \frac 1{2\omega_N^{1/N}}\bigg\}\,.
\end{equation}
Also $C_3$ is a purely geometric constant only depending on $N$, since so are $C_1,\, C_2$ and $C_F(N)$. A quick check, recalling that by assumption $C_k\geq 2^N C_F(N)$, ensures that the inequality~(\ref{quick}) is valid setting $C_{k-1}(N,D)= C_3 C_k(N,3D) D$, according to~(\ref{defCk-1}). Thus, we have proved the validity of~(\ref{inductive}) for any $(k-1)$-symmetric set defining $C_{k-1}(N,D)$ according to~(\ref{defCk-1}).\par

To conclude, we only have to check the validity of the inequality $C_{k-1}(N,D)\geq 2^N C_F(N)$. But, since we know by assumption that $C_k(N,3D)\geq 2^N C_F(N)$, this is surely true if $C_3 D\geq 1$, which in turn is true since $D\geq 2\omega_N^{1/N}$ by the isodiametric inequality and by~(\ref{defC3}).
\end{proof}

\begin{remark}\label{dependance}
As discussed in the Introduction, the constant of Theorem~\mref{main} must depend on the diameter $D$, and actually explodes when $D\to \infty$. Our construction provides an estimate with $C(N,D) \lesssim D^N$, see~(\ref{defCND}). An interesting question might be to find the sharp power of $C(N,D)$ with respect to $D$. This power cannot be lower than $\frac{N-1}N$. In fact, if the set $E$ is made by a ball of volume $1-\eps$ centered in the origin, plus a second ball of volume $\eps$ at distance $3/\eps$, then $\lambda_0(E)=2$, and $\delta(E)\approx \eps^{\frac{N-1}N}\approx D^{-\frac{N-1}N}$, where $D\approx 3/\eps$ is the diameter of $E$. Therefore, the sharp value of $C(N,D)$ must surely be at least of order $D^{\frac{N-1}N}$.
\end{remark}

\section{\label{Sect:comparison}Comparison with the other results}

In this final section, we briefly comment on the relation between our Theorem~\mref{main} and the two preceding results mentioned in the Introduction, namely, Theorem~\ref{ThBCH} by Bianchini, Croce and Henrot, and the Quantitative Inequality~(\ref{ShBar}) by Fuglede. Keep in mind that in the first paper the inequality~(\ref{ourineq}), with a purely geometric constant $C_{BCH}$ in place of $C(N,D)$, was proved for every connected set $E\subseteq \R^2$, while the inequality by Fuglede was proved with a purely dimensional constant $C_B(N)$ in place of $C(N,D)$ and for all convex sets in $\R^N$.

It is easy to see (and we are going to do it in a moment) that our result implies both the above-mentioned results. However, it is important to remind that the proof by Fuglede was particularly hard, since he could not use, as we did, the sharp quantitative inequality~(\ref{ShQII}), which is a very strong result; on the other hand, the proof of Theorem~\ref{ThBCH}, which also uses~(\ref{ShQII}), is particularly short and clear.\par

Concerning Theorem~\ref{ThBCH}, take a connected set $E$ of unit volume in $\R^2$ (keep in mind that all the inequalities are scaling-invariant, so fixing the volume is harmless). Then, there are two possibilities: either the diameter of $E$ is larger than $2\sqrt\pi$, or not. In the first case, the perimeter of $E$ is at least $4\sqrt\pi$, hence the isoperimetric deficit satisfies
\[
\delta(E) = \frac{P(E) - P(B(1))}{P(B(1))} = \frac{P(E) - 2\sqrt\pi}{2\sqrt\pi} \geq 1\,,
\]
and since the barycentric asymmetry always satisfies $\lambda_0(E)\leq 2$, we have
\[
\lambda_0(E) \leq 2 \leq 2\sqrt{\delta(E)}\,.
\]
In the second case, by Theorem~\mref{main} we have that
\[
\lambda_0(E) \leq C(2,2\sqrt\pi) \sqrt{\delta(E)}\,.
\]
Therefore, we recover Theorem~\ref{ThBCH} with the purely geometric constant $\max \{ 2, C(2,2\sqrt\pi)\}$.

Concerning Fuglede's Theorem, the situation is very similar. Indeed, take this time a convex set $E$ of unit volume in $\R^2$. Suppose first that the diameter of $E$ is larger than $N\D$, for a constant $\D$ to be precised in a moment, and in particular that the points $(0,0,0,\dots 0)$ and $(N\D,0,0,\dots 0)$ belong to $\partial E$. Therefore, by convexity we have that for every $0\leq t \leq N\D$ the section $E_t=\{ y\in\R^{N-1},\, (t,y)\in E\}$ satisfies $\H^{N-1}(E_t) \leq 1/\D$. But then,
\[\begin{split}
P(E) &\geq \int_{t=0}^{N\D} (N-1)\omega_{N-1}^{\frac 1{N-1}}\big(\H^{N-1}(E_t)\big)^{\frac{N-2}{N-1}}
\geq \D^{\frac 1{N-1}}\int_0^{N\D} (N-1)\omega_{N-1}^{\frac 1{N-1}}\H^{N-1}(E_t)\\
&= \D^{\frac 1{N-1}}(N-1)\omega_{N-1}^{\frac 1{N-1}}\,.
\end{split}\]
Since the perimeter of the unit ball is $N\omega_N^{1/N}$, the above estimate gives $\delta(E)\geq 1$ as soon as
\[
\D^{\frac 1{N-1}}(N-1)\omega_{N-1}^{\frac 1{N-1}} \geq 2N\omega_N^{1/N}\,,
\]
which is true with the choice
\[
\D := \bigg(\frac{2N}{N-1}\bigg)^{N-1} \,\omega_N^{\frac{N-1} N}\omega_{N-1}^{-1} \,.
\]
Exactly as before, since $\lambda_0(E)\leq 2$, we surely have $\lambda_0(E) \leq 2 \sqrt{\delta(E)}$. And exactly as before, applying our Theorem~\mref{main} with $D=N\D$ if the diameter of $E$ is not larger than $N\D$, we recover Fuglede's inequality with the purely dimensional constant $\max \big\{ 2, C(N, N\D)\big\}$.\par

\end{document}